\documentclass[12pt]{article}
\usepackage{amssymb}
\usepackage{graphicx}

\newtheorem{lemma}{Lemma}
\newtheorem{theorem}{Theorem}
\newtheorem{remark}{Remark}

\newtheorem{definition}{Definition}

\newcommand\R{{\mathbb R}}
\newcommand\Z{{\mathbb Z}}
\newcommand\N{{\mathbb N}}
\newcommand\m{{m \times m}}

\begin{document}

\markboth{Bose/Murray}{Polynomial decay of correlations}

\title{Polynomial decay of correlations in the generalized baker's transformation}
\author{
Christopher Bose\footnote{Department of Mathematics and Statistics, University of Victoria,
 PO Box 3060 STN CSC,
Victoria BC Canada V8W 3R4, cbose@uvic.ca}
 and Rua Murray\footnote{
Department of Mathematics and Statistics, University of
Canterbury,
 Private Bag 4800, Christchurch 8140, New Zealand,  rua.murray@canterbury.ac.nz}}

\maketitle

\begin{abstract}
We introduce a family of area preserving generalized baker's transformations acting on the
unit square and having sharp polynomial rates of mixing for H\"older data.  The construction
is geometric, relying on the graph of a single variable ``cut function''.
Each baker's map $B$ is non-uniformly hyperbolic  and
while the exact mixing rate depends on $B$, all polynomial rates can be attained.
The analysis
of mixing rates depends
on building a suitable Young tower for an expanding factor. The
mechanisms leading to a slow rate of correlation decay are especially
transparent in our examples due to the simple geometry in the construction.
For this reason we propose this class of
maps as an excellent testing ground for new techniques for the analysis of decay of
correlations in non-uniformly hyperbolic systems.
Finally, some of our examples can be seen to be extensions of certain 1-D non-uniformly expanding maps
that have appeared in the literature over the last twenty years thereby providing a unified treatment of
these interesting and well-studied examples.
\end{abstract}


\section{Introduction}\label{intro}

The analysis of non-uniformly hyperbolic systems has undergone an explosion of
activity in the last decade with a range of new techniques becoming available;
notably Young towers~\cite{Y1,Y2}, hyperbolic times~\cite{A,AA,ABV,ALP,ALP2}
and, earlier, Pesin theory for maps with singularities~\cite{KS}. The application
of this machinery on `real life' examples is often highly technical, with substantial effort being
required, for example, to isolate and analyze lower-dimensional expanding factors whose mixing properties
drive statistics for the higher-dimensional hyperbolic map. In this paper
we consider a class of two-dimensional {\em Generalized Baker's Transformations\/} (GBTs)
whose simple geometry allows, via Young towers, the extraction of sharp polynomial rates of
correlation decay on H\"older observables. The constructions and proofs are explicit and geometrically natural.
The relevant one-dimensional expanding factors are certain piecewise $C^2$, Markov,  non-uniformly
expanding maps of the unit interval, the most familiar and well-understood examples we know of
for analysis of the connection between hyperbolicity and mixing rates.

The extension from baker's to generalized baker's is easy to describe.
Specifically,
a two-dimensional map~$B$ on the unit square $S=[0,1]^2$ is determined by a \emph{cut function}
$\phi$ whose graph
$y=\phi(x)$ partitions $S$
into lower and upper pieces.
The cut function is assumed to be measurable and to satisfy $0 \leq \phi(x) \leq 1$; these
are the only constraints in the construction.
The two-dimensional dynamics are
depicted in Figure \ref{fig:gbt} and  
defined by mapping the vertical lines $\{x=0\},\{x=1\}$,
into themselves, and sending
vertical fibres into vertical fibres (the fibre over $x$ goes to part of the fibre over $f(x)$) in such
a way that areas are preserved; if we define $a=\int_0^1 \phi(t)\, dt$ then the rectangle
$[0,a]\times [0,1]$ maps to the lower part of the square
under the graph of $\phi$ and $[a,1]\times [0,1]$ maps to the
upper part.   The resulting map $B$ preserves Lebesgue measure $\m$ on the square $S$.
$B$ necessarily has a discontinuity along the vertical line $\{x=a\}$.
Clearly $B$ is hyperbolic: through each point on the square passes a contracting leaf (vertical line) and
an expanding leaf (the graph of a measurable function).  $B$ is uniformly hyperbolic if and only if the cut
function $\phi$ is bounded away from zero and one as depicted in the Figure \ref{fig:gbt}.

\begin{figure}
\begin{center}
\includegraphics[width= 0.90\textwidth ]{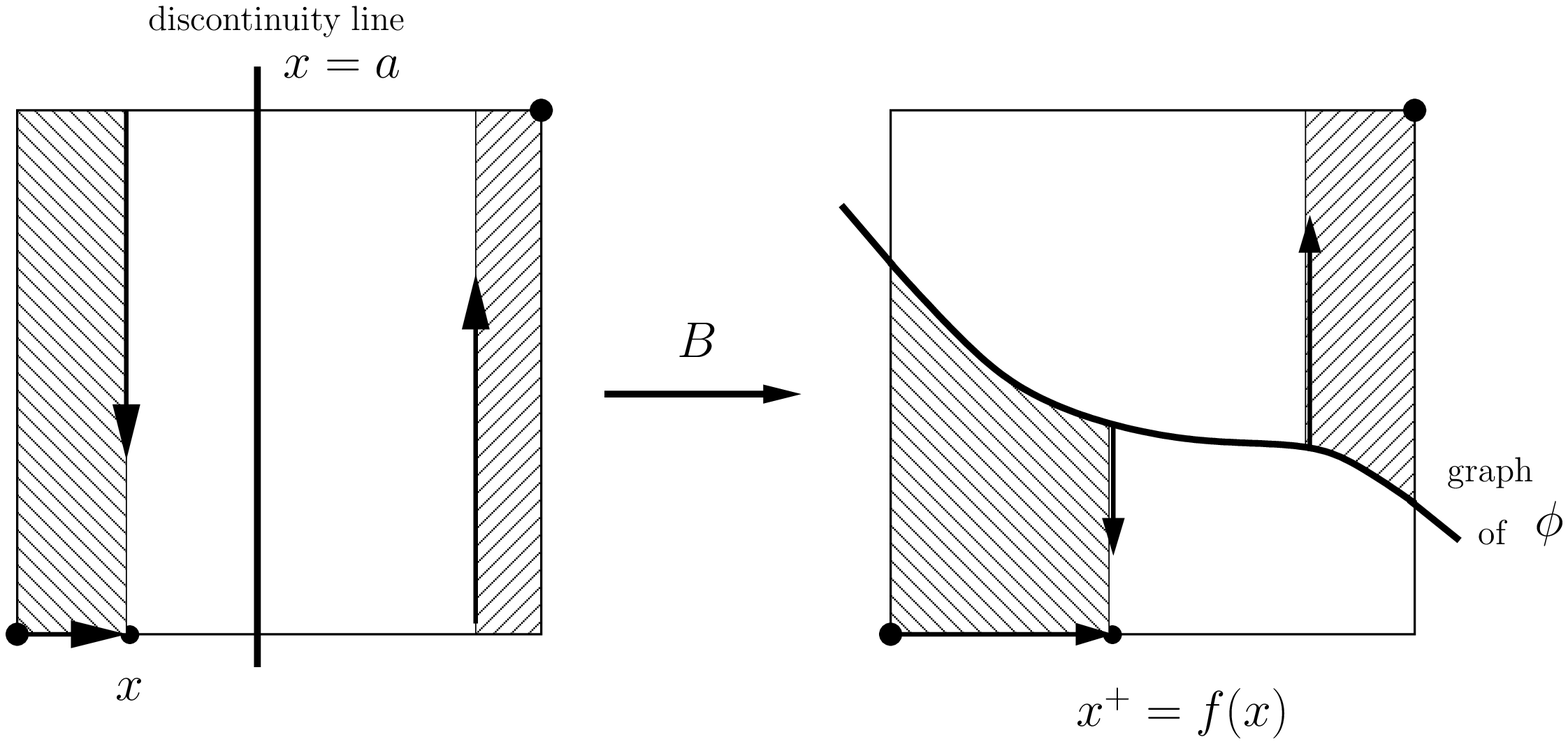}
\caption{Generalized Baker's Transformation}
\label{fig:gbt}
\end{center}
\end{figure}

When $\phi \equiv 1/2$ the map is the classical baker's transformation.

The construction was introduced in \cite{B1} where many basic dynamical properties were established.  For example, regularity  conditions on the cut allow one to conclude that $B$ is ergodic, or even Bernoulli.
Perhaps more surprisingly,  it was shown that every measure preserving transformation $T$ on a (nonatomic, standard, Borel) probability space with entropy satisfying\footnote{$\log$ will always mean the natural logarithm.  The upper bound in this inequality is not essential; any finite entropy map may be represented by a gbt, provided you allow multiple cut functions on the square.}
$0< h(T) < \log 2$ is measurably isomorphic to some GBT
on the square $S$, so in some sense, these are universal examples of measure preserving systems.

A recent literature search uncovered more than 80 articles describing generalized baker's maps,
of which the construction above represents only one possible direction. Some investigations consider only locally affine, measure preserving transformations, a minor variant of the classical example and a subcase of the present construction. There are also fat baker's tranformations -- noninvertible maps where the expansion in the unstable direction dominates contraction on vertical fibres (for example, see \cite{A-Y, Rms,T}). Generically such maps admit an absolutely continuous invariant
SRB measure. The recent article \cite{K} studies baker's transformations on non-square domains whose
expanding factors are certain $\beta-$transformations.

Our main goal in this paper is to establish sharp polynomial decay of correlation estimates for some non-uniformly hyperbolic
examples of the map $B$ above acting
on $2$-D  H\"older observables.
Obtaining sharp decay rates for multidimensional hyperbolic systems has proved to be rather difficult in general and the majority of results of this type
are in the exponential or stretched exponential class (see, for example \cite{Y1}~(Billiards),
\cite{C-Y}~(Lorentz scatterers), \cite{B-Y}~(Henon maps) for the former, and
\cite{G1}~(following\cite{V,ALP}) for the latter).
 For specific higher dimensional
families of {\em hyperbolic\/} maps
the only sharp sub-exponential results known to us are for certain billiards, in particular in a Bunimovich stadium
the rate is $O(1/n)$ (see \cite{Ma} and \cite{CZ} and references contained therein).

Although the simple geometry of our class could be viewed as artificial, it is
extremely effective for illustrating some of the obstacles (and techniques used to overcome them) that have been central to the subject in recent years.

In order to carry out our analysis, we first establish the corresponding rate-of-mixing result
on an appropriate $1$-D expanding factor $f$. This map arises naturally from the action of $B$ on
the  invariant family of `vertical fibres'; $f$ will be a piecewise monotone and continuous map of the unit
interval having
indifferent fixed points at $x=0$ and $x=1$.

Non-uniformly expanding $1-$D interval maps such as our $f$ are
currently much better understood than the corresponding multidimensional transformations.
Analysis of maps of the interval with indifferent fixed points was carried out in
 \cite{P} and references cited there.
This early work also anticipates one of the most fruitful modern approaches: the construction of a Markov Extensions or
Young towers (see  \cite{Y1,Y2}).   Indeed, we also begin by constructing a suitable
Young tower for $f$ after which, upper bounds on the rate of decay of correlation against H\"older data
are routine to obtain. In our case these rates are polynomial (the exact rate depending on which map $f$ from
the family is being considered; all polynomial mixing rates may be attained simply by the choice of
parameters leading to $f$ (and $B$)).

Recently,  \cite{C-H-V} investigates a parameterized family
of $1$-D circle maps on $[-1,1]$ proving they have polynomial mixing rates.
It turns out these maps are conjugate to certain $f$ given by our construction (see Example \ref{ex:cristadoro}). On the other hand, our examples need not be symmetric, and the $2-$D connections
we are motivated by in this paper are not investigated in \cite{C-H-V}.

Analysis of the mixing properties of $B$ proceeds by lifting the corresponding estimates for $f$ back to the square along stable fibres.  In this case, because of the simple geometry, this step is relatively simple compared to previous studies
in the literature, including the ones cited above.

Another approach to the study of non-uniformly hyperbolic maps depends on the analysis of \emph{hyperbolic times}.
In \cite{B-M} we show that, while all our examples $f$ have positive density of hyperbolic times, the first hyperbolic time function may or may not be integrable, depending on the order of tangency of the cut function to the boundary of the square at $(0,1)$ and $(1,0)$.  Indeed, it is possible to obtain sharp estimates on $m\{ h_{\sigma,\delta} > n\}$
where $m$ is Lebesgue measure and  $h_{\sigma,\delta}(x)$ denotes the first $(\sigma,\delta)$-hyperbolic time for the orbit at $x$ (see \cite{A1, A, AA, ABV} for definitions and related computations).  Analysis of hyperbolic times for our maps $f$ will not be used in this paper.

In the next section we set up the notation used throughout the paper and define our family
of baker's maps $B$. In Section~\ref{sec:youngtowers} we begin with a brief review of the Young tower construction in a form that best suits our application. In Sections~\ref{s:towers}--\ref{ss:lowerbounds} we build towers for the
$1$-D maps $f$ induced by $B$ acting on the stable leaves and establish sharp rates of correlation decay for these systems with respect to $1$-D H\"older observables.  We complete the work in Section~\ref{s:2ddecay} by lifting the $1$-D results in a natural way to identical decay estimates on the $2$-D maps $B$.

Some elementary computations (essentially calculus exercises) are gathered in Appendix 1.

\section{Generalized baker's maps}

With respect to notation from the previous section, the relevant equations are easy to derive:
$$(x,y) \mapsto (f(x),g(x,y))= B(x,y)$$
where
\begin{equation}\label{eqn:map1}\begin{array}{c}
g(x,y)=\left\{\begin{array}{ll}
\phi(f(x))\,y &\mbox{if~$x \leq a$},\\
\\
y + \phi(f(x))(1-y) &\mbox{if~$x > a$},\end{array}\right.\\
\\
\quad\textnormal{and}\quad\\
\\
x=\left\{\begin{array}{ll}
 \int_0^{f(x)}\phi(t)\,dt&
\mbox{if~$x \leq a$},\\
\\
  1- \int_{f(x)}^1 1-\phi(t) \,dt&
\mbox{if~$x > a$}.
\end{array}\right.
\end{array}
\end{equation}

Note that the function $f$ appears implicitly in
Equations (\ref{eqn:map1}).
Provided the set of $t$ where
$\phi$ takes on the
value $0$ or $1$ is of measure zero, it is easy to
see that $f$ (and hence $g$)  is uniquely defined
for every $x \in [0,1]$ (respectively, for $(x,y) \in S$).  This will be the case for
all examples in this paper.

Even without this restriction,  by construction, $B(x,y)$ is defined
by Equation (\ref{eqn:map1}) for
Lebesgue almost every point $(x,y)$ in the square $S$,
is invertible\footnote{In the usual sense of being invertible off a set of measure zero
on the square.} and preserves
two-dimensional Lebesgue measure.
For details, and a formula for $B^{-1}$ see \cite{B1}.
The sub-sigma-algebra of vertical fibres $\{x=x_0\}$ on $S$ is invariant
under\footnote{But not for
$B^{-1}$, since B maps fibres into partial fibres, in general}   $B$
and the associated (non-invertible) factor is naturally identified
with the map $f$,
a two branched, piecewise increasing map on $[0,1]$.

Define $\pi:S\rightarrow[0,1]$ by $\pi(x,y)=x$.  Then
$\pi\circ B = f\circ\pi$ encodes the factor relationship
between $f$ and $B$ and
if $\m$ denotes Lebesgue measure on $S$,
then\footnote{We adopt the standard
notation  $T_*\nu = \nu \circ T^{-1}$ for a map $T$ and measure $\nu$.}
$\pi_*(\m)=m$
is Lebesgue measure on $[0,1]$, so $f$ is also
Lebesgue measure preserving (but now on the unit interval).

From the definition of $g(x,y)$, $g(x,y)\lessgtr \phi(f(x))$ according to
whether $x\lessgtr a$; thus, the position of a point $(x,y)$ on a
vertical fibre $\pi^{-1}x$ determines the inverse
history of possible $f$-orbits, while the position $x$ specifies the future
trajectory under $x$. In this way, $B$ represents an \emph{inverse limit} or invertible cover
of the endomorphism $f$; in fact, in many cases, $B$ proves to be the \emph{natural extension}
of $f$ (for a precise treatment and conditions under which this will hold, see Section 4 of \cite{B1}).
In all our examples, $B$ will be the natural extension of $f$.

For each $n\geq 0$ the action of $B^n$ is affine on each vertical fibre,
and the skew-product character of $B$ is
emphasized through the formula:
\begin{equation}\label{eqn:Bnformula}
B^n(x,y)=(f^n(x),g_n(x,y))
\end{equation}
where $g_0(x,y)=y$ and
$$g_n(x,y) = \left\{\begin{array}{ll}
\phi(f^n(x))\,g_{n-1}(x,y)&\mbox{if~}f^{n-1}(x)\leq a\\
\phi(f^{n}(x))+(1-\phi(f^n(x)))\,g_{n-1}(x,y)&\mbox{otherwise}.\end{array}\right.$$
The geometry is illustrated in Figure \ref{fig:gbt2} for the case $n=2$.
Sometimes we'll use the notation $\tilde{\phi}={\partial_y g_1}$ for the
contractive factor on the fibres. Then $\tilde{\phi}$ depends only on~$x$, and indeed
\begin{equation}\label{eqn:lin_fibre_action}
{\partial_y g_n} = \Pi_{k=0}^{n-1}\tilde\phi(f^k(x)).
\end{equation}

Provided the cut function is smooth,
at each point $(x,y)$ in the interior of $S$ minus the vertical line $\{x = a\}$
we can compute the Jacobian matrix of $B_\alpha$ using the expressions in Equation
(\ref{eqn:map1})and the fact that $f^\prime(x) =[ \phi(f(x))]^{-1}$  for $0< x < a$
(with a similar expression for $a< x < 1$).
\begin{equation}\label{eqn:DB}
DB_\alpha(x,y) = \left\{
\begin{array}{lr}\left[
\begin{array}{cc}
\frac{1}{\phi(f(x))} & 0 \\
\\
\frac{\phi^\prime(f(x))}{\phi(f(x))}y~~&
\phi(f(x))\\
\end{array} \right] &\textnormal{if~} 0 <x < a\\
\\
\left[ \begin{array}{cc}
\frac{1}{1- \phi(f(x))} & 0 \\
&\\
\frac{\phi^\prime(f(x))}{1-\phi(f(x))}(1-y)~~&
1-\phi(f(x))\\
\end{array}\right]~~& \textnormal{if~}  a <x < 1\\ \end{array}\right.
\end{equation}
Observe
that the measure-preserving property for $B$
is again confirmed since clearly
$\textnormal{det}\,DB(x,y) = 1$.

\begin{figure}
\begin{center}
\includegraphics[width= 0.90\textwidth ]{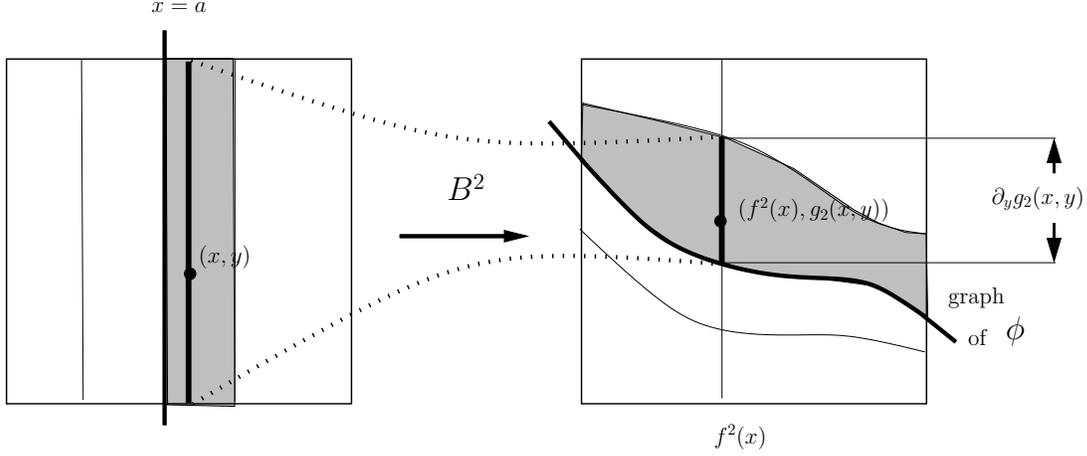}
\caption{Second iterate of a generalized baker's map acting on a vertical cylinder.}
\label{fig:gbt2}
\end{center}
\end{figure}

\subsection{The baker's family $B_{\alpha,\alpha^\prime}$}\label{sec:family}

We consider a family of generalized baker's maps indexed by two
\emph{hyperbolicity parameters}  $0< \alpha, \alpha^\prime <\infty$
through the definition of the cut function $\phi=\phi_{\alpha,\alpha^\prime}$.
Assume:

\begin{itemize}
\item  $\phi$ is  decreasing  and  $0\leq \phi \leq 1$ on $[0,1]$
\item  $\phi(0)=1$ and  there is a smooth function $g_0$ on $(0,1)$ such
that
$$\phi(t) = 1 - c_0 t^\alpha + g_0(t)$$
with $c_0 >0$ and $g_0^\prime = o(t^{\alpha -1})$  for $t$ near $0$.
\item $\phi(1)=0$ and there exists a smooth function $g_1$ on $(0,1)$ such that
$$\phi(1-t)= c_1t^{\alpha^\prime}+ g_1(t)$$
where $c_1>0$ and $g_1^\prime=o(t^{\alpha^\prime-1})$ for $t$ near $0$.
\end{itemize}

These conditions imply that the cut function $\phi=\phi_{\alpha,\alpha^\prime}$
is smooth\footnote{If $\alpha, \alpha^\prime >1$ both the cut function  $\phi$
and its derivative
extend continuously to $[0,1]$ with $\phi^\prime(0)=\phi^\prime(1)=0$. }
 on $(0,1)$ with
continuous extension to $[0,1]$ and that $0<\phi(t) <1$ for all $0< t<1$.
It follows that the map $f$ defined by Equation \ref{eqn:map1} is piecewise
strictly increasing and expanding ($f^\prime \geq 1$)
with respect to the intervals $[0,a]$ and $[a,1]$. Each branch is surjective and $C^2$
when restricted to the interior of its domain ($(0,a)$ or $(a,1)$ respectively).

\subsection{Example}\label{ex:alpha=1}
Set $\alpha=\alpha^\prime =1,~c_0=c_1=1$ and $g_i \equiv 0$. Then $\phi(x)=1-x$
and $a=1/2$.  The map $B$ is non-uniformly hyperbolic, with lines of fixed points along
$\{x=0\}$ and $\{x=1\}$

The integrals defining $f$ in (\ref{eqn:map1})
are easily computed, yielding
$$f(x) = \left\{\begin{array}{ll} 1-\sqrt{1-2\,x}~~
		&\mbox{if $x<\textstyle{\frac{1}{2}}$},\\
		\sqrt{2\,x-1}&\mbox{if $x>\textstyle{\frac{1}{2}}$}.
	\end{array}\right.$$
We emphasize again that $f$ is a
measure-preserving circle endomorphism on
$[0,1)$ with a discontinuity in
$f^\prime$ at the single point $a={1 \over 2}$, and
a neutral fixed point at $x=0$, but in this case, with quadratic order
of contact.
Thus the example does not fit into the usual picture for maps
with indifferent fixed points (eg: \cite{P,Y2} or the AFN maps of \cite{Z}) . In fact, the
branches of $f$ do not have bounded distortion in
the usual sense, since $f'(x)\rightarrow\infty$
as $x\rightarrow\frac{1}{2}$. Observe, however, that the slow escape of
mass in the neighbourhood of the neutral point $x=0$ is perfectly balanced by
a very small rate of arrival in these intervals
(for example, $f^{-1}\left([0,\epsilon)\right)\setminus[0,\epsilon)
=[\frac{1}{2},\frac{1}{2}+O(\epsilon^2))$). It is this
mechanism which allows all maps in our family
to have a finite invariant measure, despite the fixed
points being only weakly repelling.

This example has been studied previously in the literature.
It is described in \cite{Z} where it is attributed to M.\/ Thaler.
\cite{R} studied the baker's map $B$ associated to this $\phi$, proving that
it is isomorphic to a Bernoulli shift by showing that the partition into regions above
and below the cut function was weakly-Bernoulli (i.e.\ satisfying a certain mixing rate on cylinders; see Section 8 of \cite{R}).
The map $f$ also appears in
Alves-Ara\'ujo \cite{AA} as an example having a non-integrable first hyperbolic
time.

\subsection{Example}\label{ex:cristadoro}
Set $\alpha^\prime=\alpha \in (0,\infty)$, $c_0=c_1=2^{\alpha-1}$ and $g_i \equiv 0$.
Let $\phi=\phi_\alpha$ denote the cut function. Then
an easy computation shows that $a = \int \phi_\alpha = 1/2$ and
\begin{equation}
\label{eqn:phi}
\phi=\phi_\alpha(x)=\left\{\begin{array}{ll}
1-2^{\alpha -1}x^\alpha~~ &\mbox{if~$x \leq {1 \over 2}$},\\
2^{\alpha -1}(1-x)^\alpha~~ &\mbox{if~$x \geq {1 \over 2}$},\end{array}\right.
\end{equation}
Let $f_\alpha$ be the resulting $1$-D expanding map. It is straightforward to check that
this map is conjugate\footnote{Via the affine conjugacy $x \rightarrow \frac{1+x}{2}$;
the parameters satisfy $\gamma=\alpha +1$.} to the family of examples
treated in \cite{C-H-V}.  The results in the next theorem recover decay rates
obtained in that work.

\subsection{Statement of the main results}

\begin{theorem}\label{th.summary}\textnormal{[Decay of Correlations for $f$ and $B$]}
Let $\phi$, $f$ and  $B$ be as prescribed above and
set $\gamma= \max\{ \alpha, \alpha^\prime\}$.
\begin{enumerate}
\item If $\varphi$ is essentially bounded and measurable and $\psi$ is H\"older continuous on
$[0,1]$ then
$$\left|\int_0^1\varphi\circ f^n\,\psi\,dm
-\int_0^1\varphi\,dm\,\int_0^1\psi\,dm\right|=O(n^{-1/\gamma}).$$
\item
If $\varphi$ and $\psi$ are both H\"older continuous on $S$ then
$$\left|\int_S\varphi\circ B^n\,\psi\,d\m
-\int_S\varphi\,d\m\,\int_S\psi\,d\m\right|=O(n^{-1/\gamma}).$$
\end{enumerate}
If $\phi$ is symmetric (i.e.: $\phi(1-t) = 1- \phi(t)$) then in both cases
the rates above are sharp, even for Lipschitz continuous data.

\end{theorem}

Precise versions of the first part are given in Theorems \ref{endomorphism} and \ref{t.lower_function},
while Theorem \ref{th.2ddecay} handles the second part.

\section{Young Towers}\label{sec:youngtowers}

In order to proceed,
we outline the machinery developed in~\cite{Y1,Y2}
for analysis of non-uniformly hyperbolic dynamics using an abstract tower
extension.

The construction begins with a set $\Delta_0$, along with
a $\sigma$-algebra $\mathcal{B}_0$ of subsets  of
$\Delta_0$ and a finite measure
$\mu_0$ on $\mathcal{B}_0$.  A ($\mathcal{B}_0$-measurable)
\emph{return time function\/} $R:\Delta_0 \rightarrow \Z^+$ defines a \emph{tower}
$$\Delta := \{(z,l) ~:~ z \in \Delta_0,~ l\in \Z,~ 0 \leq l < R(z) \}.$$
Regarding $\Delta$ as a subset of $\Delta_0 \times \Z^+$,
denote $$\Delta_l = \Delta \cap (\Delta_0 \times \{l\})$$
---the $l$th level of the tower (when there is no ambiguity, we allow
the identification $\Delta_0\equiv(\Delta_0\times\{0\}$)).
The measure~$\mu_0$ is extended to the tower $\Delta$ by defining
$A \times \{l\} \subseteq \Delta_l$ to be measurable if  $A \in \mathcal{B}_0$
and setting $\mu(A \times \{l\}) := \mu_0(A)$.  Naturally, $\mu|_{\Delta_0}=\mu_0$.
$\mu$ is called the \emph{reference measure} on the tower $\Delta$.

Let $\{\Delta_{0,i}\}$ be a
measurable and countable partition of $\Delta_0$ such
that $R$ is constant on each atom of the partition.

\begin{remark}\label{rk:almost_everywhere}
We emphasize at this point that the tower construction is
carried out in the
measurable category, so for example, the term \emph{partition}
above refers to
a collection of measurable subsets which are
disjoint mod zero and whose union is $\Delta_0$ mod zero with respect to $\mu$.
Similarly, $R$ is understood to be $\mathcal{B}-$measurable
and constant $\mu-$a.e. on each $\Delta_{0,i}$.
\end{remark}

A map  $F:\Delta \rightarrow \Delta$ is provided on the tower
such that $F(z,l) = (z, l+1)$ if $l < R(z)-1$ and $F(z, R(z)-1) \in \Delta_0$.
Hence $F^R: \Delta_0 \rightarrow \Delta_0$ is the
first return map to $\Delta_0$, and $R$ can be extended to a function $\hat{R}$
on $\Delta$ as the first passage time to $\Delta_0$ ($\hat{R}(z,l)=R(z)-l$).
$F$ carries the partition of $\Delta_0$
into a partition $\eta$ of the tower:
$\Delta_{l,i} = \{(z,l)\in \Delta~:~z \in \Delta_{0,i} \}$
and one assumes that the partition generates, in the sense
that $\bigvee_{j=0}^\infty F^{-j}\eta$
separates the points of $\Delta$.
For our purposes, suppose also that $F^R: \Delta_{0,i} \rightarrow \Delta_0$
is bijective ($\mu$-a.e.) for each $i$, and both $F^R|_{\Delta_{0,i}}$
and its inverse are nonsingular with respect to $\mu$.
The Jacobian of this return map with respect to $\mu$ will be denoted by $JF^R$ and
on each $\Delta_{0,i},~JF^R >0$, again by assumption.

Regularity of functions on $\Delta$ is measured with respect to a \emph{separation
time\/} on the tower. Roughly speaking, a H\"older function will give similar values to
$x$ and $y$ if the first $n$ terms of the orbits of $x$ and $y$ visit the same
sequence of atoms of $\eta$ as one another\footnote{From this point on we simplify notation
and write $x$ instead of $(z,l)$ for points in the tower. }.
The measure of separation $s$ is defined as follows:
\begin{definition}\label{def:separation}
In the notation established above:
\begin{itemize}
\item if $x,y$ are in different atoms of $\eta$, $s(x,y)=0$;
\item if $x,y\in\Delta_{0,i_0}$, set $s(x,y)$ to be the
minimum $n>0$ such that $(F^R)^n(x), (F^R)^n(y)$ lie in different atoms $\eta$;
\item if $x,y\in\Delta_{l,i}$ put $s(x,y):=s(F^{\hat{R}}(x),F^{\hat{R}}(y))-1=s(x^\prime,y^\prime)$
where $x^\prime,y^\prime\in\Delta_{0,i}$ are the first unique preimages of $x,y$ in $\Delta_0$
under iteration by $F^{-1}$.
\end{itemize}
\end{definition}

Clearly $s<\infty$ since $\bigvee_{j=0}^\infty F^{-j}\eta$ separates points.
In fact, $s$ distinguishes two classes of H\"older functions:
for $0<\beta< 1$ 
$$C_{\beta}(\Delta) = \{\psi:\Delta \rightarrow \R : \exists \, c_\psi {\rm ~s.t.~}
\forall x,y \in \Delta, |\psi(x) - \psi(y)| \leq c_\psi \beta^{s(x,y)} \}$$
and
$$\begin{array}{rl} C_{\beta}^+(\Delta) &=  \{\psi:\Delta \rightarrow [0,\infty) : \exists \,c_\psi {\rm ~s.t.
~for~each~} l,i {\rm ~either~} \psi \equiv 0  {\rm~on~} \Delta_{l,i}\\&
{\rm ~or~} \psi >0 {\rm~on~} \Delta_{l,i}  {\rm ~and~}
|\frac{\psi(x)}{\psi(y)} -1| \leq
c_\phi \beta^{s(x,y)} ~ \forall \, x,y, \in \Delta_{l,i}\}. \end{array}$$

The regularity of $F$ is described by a H\"older condition on
the Jacobian of the maps
$(F^R|_{\Delta_{0,i}})^{-1}: \Delta_0 \mapsto \Delta_{0,i}$
(anticipating their appearance in the transfer operator for $F^R$):
we suppose there exist $0<\beta <1$ and $C$ such that
\begin{equation}\label{eq:reg}
\bigg| \frac{JF^R(x)}{JF^R(y)} - 1 \bigg| \leq C\beta^{s(F^R(x), F^R(y))}, ~ \forall~i,~ \forall
~x,y \in \Delta_{0,i}.
\end{equation}

We adopt the conventional notation
for asymptotics of sequences: $x_n = {\rm O}(y_n)$ means there exists
a constant $C < \infty$ such that for all large $n$, $x_n \leq Cy_n$ and
$x_n \approx y_n$ if both $x_n = {\rm O}(y_n)$ and $y_n = {\rm O}(x_n)$.

\begin{theorem}\label{thm:lsy}
\textnormal{[Young's Theorem~(part of Theorems 1-3) in \cite{Y2}]}
Assume the setting and notation above (including the regularity
condition~(\ref{eq:reg})). Assume also
that $\int_{\Delta_0} R\,d\mu < \infty$ and
that gcd$\{R_i\}=1$ where $R_i := R|_{\Delta_{0,i}}$.  Then,
\begin{itemize}
\item $F$ admits an absolutely continuous (w.r.t. $\mu$) invariant
probability measure $\nu$ on $\Delta$
with $\frac{d\nu}{d\mu} > 0$. Moreover, the system $(F,\nu)$ is exact.
\\

\noindent
Furthermore, if there is a constant $\gamma >0$ such that
$\mu\{\hat R >n\} ={\rm O}(n^{-\gamma})$ then:\\

\item for a probability measure $\lambda$ with $\frac{d\lambda}{d m} \in C_\beta^+(\Delta)$ we have
$$|F^n_*\lambda - \nu | = {\rm O}(n^{-\gamma});$$
\item for each $\varphi \in L^\infty$ and $\psi \in C_\beta(\Delta)$ we have
$$\bigg | \int_\Delta \, (\varphi \circ F^n) \psi\, d\nu -
\int_\Delta \varphi \, d\nu  \int_\Delta \psi \, d\nu \bigg| \leq |\varphi|_\infty\,C_\psi\,n^{-\gamma}$$
where $C_\psi<\infty$ depends on $\psi$ and the tower.
\end{itemize}

\end{theorem}

Observe that $\mu\{\hat R > n\} = \sum_{l>n}\mu(\Delta_l)$ so the asymptotics
above are exactly the decay rate of the mass in the top of the tower.  The theorem
shows that these rates simultaneously control (i) the relaxation rates
of non-invariant measures (with suitable H\"older densities)
under iteration by $F$ to the invariant measure,
and (ii) the rate of correlation decay with respect to the invariant
measure over a large class of regular functions.
(The decay of correlation statement is slightly different to~\cite[Theorem~3]{Y2},
and follows immediately from the speed of convergence to equilibrium for
measures---see~\cite[Section~5.1]{Y2}.)

\section{Towers for $f$}\label{s:towers}

For the rest of this article we will assume that the values $\alpha, \alpha^\prime \in (0,\infty)$,
constants $c_i>0$ and functions $g_i$ defining $\phi$ have been chosen subject to the conditions
in Section \ref{sec:family}, and the baker's map $B$ and interval map $f$ are therefore determined.
We now show how the abstract tower construction applies to
our map $f$ .

Note that $f$ admits a period--$2$ orbit $\{x_0,x_0^\prime\}$ since
$f^2$ is a four-branched, piecewise continuous and onto map. We may assume
that\footnote{Let $x_0$ be the fixed point for $f^2$ on the second
branch.} $x_0 < a$ and $x_0^\prime > a$. To illustrate using
Example \ref{ex:alpha=1}, we have
$x_0=\sqrt{2}-1$ and $x_0^\prime = 2 - \sqrt{2}$.

Let $\Delta_0=[x_0,x_0^\prime)$. Let $\{x_n\}$ be
defined under the left branch of $f$ (recursively) by $f(x_n)=x_{n-1}$. Put
$J_n=[x_{n+1},x_{n})$.  A parallel construction under the right branch
yields a sequence $x_n^\prime$ and intervals $J_n' =[ x_n^\prime, x_{n+1}^\prime)$ in
$[x_0^\prime,1]$. Finally, put $I_{n+1}=f^{-1}(J_{n})\setminus J_{n+1}$
(and similarly for
$\{I_n'\}$).
Observe that the half open subintervals $I_n \subseteq (a, x_0^\prime)$
while $I_n^\prime \subseteq [x_0,a)$.
Let $R$ denote the \emph{first return time} function to $\Delta_0$. Under
the map $f$,
we have
\begin{equation}
I_k\rightarrow J_{k-1} \rightarrow J_{k-2}\rightarrow\cdots\rightarrow J_0
\rightarrow \Delta_0,
\end{equation}
and similarly for the $I_n^\prime$ and $J_n^\prime$ intervals. Note that
each application in the composition is injective and onto. Thus,
$R(x)=k+1$ when $x\in I_k^{(\prime)}$; moreover,
$f^R$ maps bijectively to $\Delta_0$ from each $I_k^{(\prime)}$.
To summarize, in the terminology of the
previous section,  the base of the tower is taken to be $\Delta_0$,
with Borel sets and Lebesgue measure $m$; $\Delta_0$
is partitioned by two infinite sets
of half-open intervals $\Delta_{0,i}=I_i\times\{0\}$ and
$\Delta_{0,i}^\prime= I_i^\prime \times\{0\}$. Then,
$R|_{\Delta_{0,i}^{(\prime)}}
= i+1$ ($i\geq 1$) and the tower is
$$\Delta = \cup_{i=1}^\infty\cup_{l=0}^i(\Delta_{l,i}\cup \Delta_{l,i}^\prime),$$
where $\Delta_{l,i}^{(\prime)}:= \Delta_{0,i}^{(\prime)}\times\{l\}$, embedding
the tower in $\Delta_0 \times \Z^+$.

The tower map is
$$F(x,l) =\left\{\begin{array}{ll} (x,l+1) &\mbox{if $l<R(x)-1$},\\
		(f^R(x),0)&\mbox{if $l=R(x)-1$ and $R=R(x)$}.\end{array}\right.$$

To establish the regularity condition~(\ref{eq:reg}) and estimate the distribution
of the tail of~$R$, we use the following asymptotics on $x_n$ and intervals $I_n$ and $J_n$.

\begin{lemma}\label{l.meas}
\begin{enumerate}
\item[(i)]  $x_n\approx \left(\frac{1}{n}\right)^{1/\alpha}$;  $1-x_n^\prime \approx
\left(\frac{1}{n}\right)^{1/{\alpha^\prime}}$
\item[(ii)] $m(J_k)\approx (\frac{1}{k})^{1+1/\alpha}$;
$m(J_k^\prime)\approx (\frac{1}{k})^{1+1/{\alpha^\prime}}$
\item[(iii)] for $x \in I_k, I_k^\prime$, $f^\prime(x) \approx k$
\item[(iv)] $m(I_k)\approx (\frac{1}{k})^{2+1/\alpha}$;
$m(I_k^\prime)\approx (\frac{1}{k})^{2+1/{\alpha^\prime}}$
item[(v)] if  $\rho>0$ then $x_k-x_{k+n}\approx x_k\frac{n}{k}$ when
$n\leq\rho\,k$.
\end{enumerate}
\end{lemma}

\begin{proof} See Appendix 1. \end{proof}

The separation function $s$ is given by Definition \ref{def:separation} with respect to
the partition $\eta$ of $\Delta$, although we emphasize that $\Delta_{l,i}^\prime$
and $\Delta_{l,i} \neq \Delta_{l,i}^\prime$ are understood to be different atoms in
$\eta$ even though the value of the return time function $R$ is the same on both intervals.

\begin{lemma}\label{l:metric} There exists a
constant $\beta=\beta(f) <1$ such that if $x, y \in \Delta_0$ and
$s(x,y)=n$ then $|x-y| \leq \beta^n$
\end{lemma}
\begin{proof} Set $\beta := \min\left\{ [f^\prime(x_0^\prime)]^{-1},
[f^\prime(x_0)]^{-1}\right\} <1$
and observe that on the set $\Delta_0,~f^\prime \geq \beta^{-1}>1$,
and hence $(f^R)^\prime\geq \beta^{-1}$ (recall $f^\prime\geq 1$ everywhere).
Therefore, if $x,y$ lie in a common atom $\Delta_{0,i}^{(\prime)} \subseteq  (f^R)^{-1}[x_0,x_0^\prime]$
with $x= (f^R)^{-1}(x^\prime),~y=(f^R)^{-1}(y^\prime)$ then $|x-y|\leq \beta$. The
result follows by induction on $i \leq n$.
\end{proof}

\begin{lemma}[Uniform distortion] \label{l:distortion}
Let $y,z\in\Delta_0$ and suppose that $s(y,z)\geq 1$.
Then there is a constant $D>1$ (depending on $f$ but not $y,z$) such that
$$\left|\frac{{f^R}'(y)}{{f^R}'(z)}-1\right|\leq \frac{D\,(D-1)}{m(\Delta_0)}\,
|f^R(y)-f^R(z)|.$$
\end{lemma}

\begin{proof} See Appendix 1. \end{proof}

\begin{remark}\label{rem:nonsingF}
The ambient measure $\mu_0$ from the abstract
tower construction is chosen to be Lebesgue measure
$m|_{[x_0, x^\prime_0)}$. Its lift to the tower
$\Delta$ under $F$ is the product of
Lebesgue measure  with  counting measure restricted to $\Delta$,   which
we will denote by $m_\Delta$.
Note, however, that since
$m$ is invariant for $f$, $m|_{\Delta_0}$ is
$f^R$--invariant on $\Delta_0$. Since
$F^R(x)= f^R(x)~\forall~ x \in \Delta_0$, $m_\Delta$
is $F-$ invariant on the tower.  Therefore $F^R$ and its inverse satisfy
the required
nonsingularity assumption
as maps between $\Delta_{0,i}^{(\prime)}$ and $\Delta_0$.
\end{remark}

\section{Mixing rates I -- upper bounds for the tower map $(F,\Delta)$}\label{ss:upperbounds}

Recall that  $m_\Delta$ denotes the
product of Lebesgue measure with counting measure
on the tower $\Delta$.

\begin{theorem} \label{tower}
Fix $f$ be as in the previous section and
any $\beta \geq \beta(f)$ as in Lemma \ref{l:metric}.
Set $\gamma=\max\{ \alpha, \alpha^\prime\}$.
Then
\begin{enumerate}
\item $m_\Delta(\Delta)=1$ and $m_\Delta$ is the unique
absolutely continuous  $F-$invariant probability measure on $\Delta$.
Moreover,
the system $(F,m_\Delta)$ is exact, hence ergodic and mixing.
\item For each absolutely continuous probability
measure $\lambda$ such that
$\frac{d\lambda}{dm_\Delta} \in C_\beta^+$ we have
$$|F^n_*\lambda - m_\Delta| = O(n^{-\frac{1}{\gamma}})$$
\item For every $\varphi \in L^\infty(\Delta)$
and $\psi \in C_\beta(\Delta)$  we have
$$\bigg| \int \varphi \circ F^n ~\psi \, dm_\Delta - \int \varphi \, dm_\Delta \int \psi \,
dm_\Delta \bigg| \leq |\varphi|_\infty\,C_\psi\,n^{-\frac{1}{\gamma}}$$
where $C_\psi<\infty$ depends only on $\psi$ and $f$.
\end{enumerate}\end{theorem}

\begin{proof} (1)  Since $F$ is non-singular with respect to $m_\Delta$ (see Remark~\ref{rem:nonsingF}), Lemmas~\ref{l:metric} and~\ref{l:distortion} give the regularity estimate (\ref{eq:reg}) on
the tower map $F$  with $\beta:= \beta(f)$, $D:=D(f)$ and
$C:= \frac{D(D-1)}{m(\Delta_0)}$
(one simply observes that $|f^R(y)-f^R(z)| \leq \beta^{s(f^R(y),f^R(z))}$
and that $F^R = f^R$). It follows that (\ref{eq:reg}) is satisfied
for every $\beta \geq \beta(f)$.
Next, using Lemma \ref{l.meas} we can estimate
$$\int_{\Delta_0} R(x)\, dm(x) = \sum_{k=1}^\infty (k+1)m(I_k \cup I_k^\prime)\leq K
\sum_{k=1}^\infty (k+1)\bigg(\frac{1}{k}\bigg)^{2+ \frac{1}{\gamma}} < \infty$$
for some constant $K$.
Moreover, this shows
$$\int_{\Delta_0} R(x)\, dm(x) =
O\left(\sum_{k=1}^\infty \bigg(\frac{1}{k}\bigg)^{1+ \frac{1}{\gamma}}\right)$$

Finally, we note that the values
taken by the return time function are $R= 2,3, \dots$ so the $\gcd$ condition
in Theorem \ref{thm:lsy} also holds.
Applying the theorem to our
tower yields an invariant measure $\nu$ on $\Delta$ equivalent
(i.e. mutually absolutely continuous) to $m_\Delta$.
Since the latter is already $F-$invariant, we claim $m_\Delta=\nu$.

To confirm this, note that since $\nu$ is ergodic we can decompose $m_\Delta=p\,\nu+(1-p)\,\nu_\perp$
where $\nu$ and $\nu_\perp$ are mutually singular. If there is a set $A$ such that
$\nu_\perp(A)>0$ but $\nu(A)=0$ then $m_\Delta(A)=0$ since $m_\Delta$ and $\nu$ are
equivalent measures. Hence $(1-p)=0$, establishing the claim.

Conclusions (2)-(3) of Theorem \ref{thm:lsy} also apply since
$$m_\Delta(\hat R > n)= \sum_{l>n}m_\Delta(\Delta_l) =  \sum_{l>n} (l-n)\, m(I_l \cup I_l^\prime) \approx \bigg( \frac{1}{n}\bigg)^{\frac{1}{\gamma}} $$
(by Lemma \ref{l.meas}).
\end{proof}

\section{Mixing rates II -- upper bounds for the factor map $(f,[0,1]$)}

The tower $(F,\Delta)$ provides a
representation for the dynamics of $f$
oriented around the {\it induced
transformation\/} $f^R$ of first returns to $\Delta_0$.
In order to interpret the mixing results of Theorem~\ref{tower}
in terms of the original map $f$ we first
extract $f$ as a factor of $F$.

For $(x,l) \in \Delta$ define
$$\Phi(x,l)=f^l(x)$$
(For convenience set $f(a)=0$ which is consistent with viewing $f$ as a
continuous circle endomorphism).
Now:

\begin{itemize}
\item $\Phi|_{\Delta_0} \equiv \mbox{id}_{[x_0,x^\prime_0)}$
\item For $l >0$ , $\Phi$ maps $\Delta_l$ injectively onto $[0,x_0)\cup[x^\prime_0,1)$
\item $\Phi^{-1}(J_k) =
\bigcup_{l=1}^\infty I_{l+k}\times\{l\}$ (with a similar equality for $\cdot'$)
\item There exists a $D^\prime$ such that for
all $l<i$, if $A\subseteq I_{0,i}\times\{l\}= \Delta_{l,i}$
then
\begin{equation}\label{eqn:bddDofftower}
{D^\prime}^{-1}\leq\frac{m(A)}{m(I_i)}\,\frac{m(J_{i-l})}{m(\Phi(A))}\leq D^\prime
\end{equation}
(with a similar inequality for $\cdot'$).
\item The semi-conjugacy property:
$$\begin{array}{rl}\Phi\circ F(x,l) &= \left\{\begin{array}{ll}
\Phi(f^{l+1}(x),0) &\mbox{if $x\in\Delta_{0,l}$},\\
\Phi(x,l+1) &\mbox{if $x \in \Delta_{0,k},~ k>l$} \end{array}
\right. \\[1.5em]
&=f^{l+1}(x) =f(f^l(x))=f\circ\Phi(x,l).\end{array}$$
\item That $\Phi_*m_\Delta = m_{[0,1]}$.  This computation can be done by
bare hands, or one can use the $F$--invariance of $m_\Delta$ as follows:
From Theorem~\ref{tower} we know that $f_*\Phi_*m_\Delta=\Phi_*F_*m_\Delta=\Phi_*m_\Delta$,
and since $(F,m_\Delta)$ is ergodic, $(f,\Phi_*m_\Delta)$ is ergodic. Moreover,
$m_{[0,1]}\ll\Phi_*m_\Delta$ by the distortion relation~(\ref{eqn:bddDofftower}),
so equality of the two measures follows by the same argument as in the proof
of Theorem~\ref{tower}(1).
\end{itemize}

Now suppose $\psi$ is $\zeta$--H\"older continuous\footnote{Meaning,
$|\psi(x) - \psi(y)| \leq C |x-y|^\zeta$, for some $C,~\zeta>0$ and all $x,y$.} as a function on $[0,1]$,
and denote $\hat \psi:=\psi\circ\Phi$ (the natural lift to $\Delta$).

\begin{lemma} \label{lift} Let $\beta = \beta(f)$ from Lemma \ref{l:metric}. If $\psi$ is a  $\gamma$--H\"older
then  $\hat\psi \in C_{\beta_0}(\Delta)$, where $\beta_0= \beta^{\gamma}$.
\end{lemma}

\begin{proof}
We need to check the regularity condition on $\hat\psi$. First,
if $(x,l), (y,k)$ are not on the same level of the tower, then $s((x,l),(y,k))=0$
and we estimate (for any choice of $\beta$)
$$|\hat\psi(x,l) - \hat\psi (y,k)| \leq 2\,|\psi|_\infty \beta^0$$
In fact, the same inequality holds also whenever $s((x,l),(y,l))=0$ on the same level
of the tower in which case $c_\psi = 2\,\|\psi\|_\infty$ will do the job.
Now suppose $s((x,l),(y,l))=n>0$.
Then, with $C$ and $\zeta>0$ from the H\"older condition on $\psi$ and applying
Lemma \ref{l:metric} we obtain
$$\begin{array}{rl} |\hat \psi(x,l) - \hat \psi (y,l)| &= |\psi(f^l(x))-\psi(f^l(y))|\\
&\leq C |f^l(x) - f^l(y)|^\zeta\\ &\leq C|(F^R(x))-(F^R(y))|^\zeta\\
&\leq C \beta^{(n-1)\zeta} = C \beta^{-\zeta} (\beta^\zeta)^n
\end{array},$$
where we have used $s(F^R(x),F^R(y))= n-1$.
Therefore it suffices to take $c_{\hat\psi}=\max\{C \beta^{-\zeta}, 2\,|\psi|_\infty\}$ and
$\beta_0 = \beta^\zeta$ in the definition of $C_{\beta_0}(\Delta)$.
\end{proof}

\begin{theorem}\label{endomorphism}
Let $\gamma = \max\{ \alpha, \alpha^\prime\}$.
\begin{enumerate}
\item The system $(f, m)$ is exact and hence $B$ acting
on $S$ is a K-automorphism.
\item If $d\lambda = \psi\, dm$ is any absolutely continuous probability
measure with $\psi$ H\"older continuous, then
$$|f_*^n \lambda - m| = O(n^{-\frac{1}{\gamma}}).$$
\item
If $\varphi \in L^\infty[0,1]$ and
$\psi:[0,1] \rightarrow \R$ is H\"older continuous, then

$$\bigg| \int_0^1 \varphi\circ f^n\,\psi\, dm -\int_0^1 \varphi \, dm
\int_0^1 \psi \, dm \bigg| \leq |\varphi|_\infty\,C_\psi\,n^{-\frac{1}{\gamma}}$$
where $C_\psi<\infty$ depends only on~$\psi$ and $f$.\end{enumerate}
\end{theorem}

\begin{proof} Denote again by $m_\Delta$ Lebesgue measure on the
tower.
Since $(f, m)$ is a factor of the exact system $(F, m_\Delta)$, it is also exact,
and hence its natural extension $B$ on $S$ is a $K-$automorphism.
Next we may assume $\zeta \leq 1$ in the H\"older condition,
so
$\beta^{\zeta} \geq \beta$.
Finally, observe the elementary identity
$$\int_{[0,1]} q(x) dm(x) = \int_{[0,1]} q(x) d\Phi_* m_\Delta = \int_\Delta \hat q dm_\Delta$$
Now an application of Lemma \ref{lift}, combined
with the decay of correlations result in Theorem~\ref{tower}, using the value of $\beta^\zeta
\geq \beta(f)$ yields the result.
\end{proof}

\section{Mixing rates III -- lower bounds for the factor map $(f,[0,1])$
}\label{ss:lowerbounds}

The upper bounds on speed of convergence to equilibrium and correlation
decay obtained in Theorem~\ref{endomorphism} in parts (2) and (3)
are in fact sharp in many situations.

We first treat the measure decay result, where
lower bounds on the decay rate are effectively determined by the behaviour
of initial densities in the neighbourhoods of the indifferent fixed points
at $0$ and $1$.  The argument is quite intuitive.

We say a probability
measure {\em $\lambda$ is separated from $m$ at $x$} if either
$$\limsup_{\epsilon\rightarrow 0^+}{\textstyle
\frac{\lambda(x-\epsilon,x+\epsilon)}{m(x-\epsilon,x+\epsilon)}} < 1
\mbox{~or~}\liminf_{\epsilon\rightarrow 0^+}{\textstyle
\frac{\lambda(x-\epsilon,x+\epsilon)}{m(x-\epsilon,x+\epsilon)}} > 1.$$

\begin{theorem}
\label{t.lower_measure}\textnormal{ [Sharp decay rates for measures] }
 Let $\lambda\ll m$ be a probability measure on $[0,1]$
such that $\varphi:=\frac{d\lambda}{dm}\in L^\infty$.
If $\lambda$ is separated from $m$ at $0$
then for $n\in\N$, $|{f_*}^n\lambda-m|\geq c\,n^{-1/\alpha}$
($c>0$ is a constant depending on $\lambda$ and $\alpha$).
If $\lambda$ is separated from $m$ at 1, the same result holds
with $\alpha$ replaced by $\alpha^\prime$.
\end{theorem}

While it is possible for correlations to decay faster than the
rate specified in Theorem \ref{endomorphism}, $L^\infty$ initial densities which differ
slightly from their equilibrium value at the indifferent fixed points
{\em must decay slowly\/}.

\begin{proof}  We consider the case of a measure $\lambda$
separated from $m$ at zero.  The proof of the second part of the theorem is
identical.

Suppose first that $\limsup_{x\rightarrow 0}\frac{\lambda[0,x]}{x}<1$.
Let $\epsilon,\delta>0$ be such that $\lambda[0,u) < (1-\delta)\,u$ for all
$u\in(0,\epsilon)$. Write $f^{-n}[0,u) = [0,v)\cup A_n$ where $f^n(v)=u$
and $A_n$ is a union of $2^n-1$ subintervals of $(v, 1]$.
Then, ${f_*}^n\lambda[0,u)\leq \lambda[0,v)
+ |\frac{d\lambda}{dm}|_\infty\,m(A_n)$.
Since $m$ is $f$ invariant, $u=m[0,u]=m\circ f^{-n}[0,u] = v+m(A_n)$. Now
let $u=x_k$, where $k$ is large enough that $x_k<\epsilon$ and $k\geq n$. Then,
$v=x_{k+n}$ and
$${f_*}^n\lambda[0,u)\leq (1-\delta)\,v + \left|\textstyle\frac{d\lambda}{dm}
\right|_\infty\,(u-v)
\leq (1-\delta)\,x_k +
\left|\textstyle\frac{d\lambda}{dm}\right|_\infty\,c_2\,x_k\,
\textstyle\frac{n}{k}$$
(where the finite $c_2$ is chosen corresponding to $\rho=1$ in
Lemma~\ref{l.meas}~(v)).
Now, choose $N\in\N$
such that $\frac{|\frac{d\lambda}{dm}|_\infty\,c_2}{N}<\frac{\delta}{2}$
and $x_N<\epsilon$. Using $u=x_k=x_{nN}$,
$${f_*}^n\lambda[0,x_{nN})\leq (1-\delta)\,x_{nN} +
\textstyle\frac{\delta}{2}\,x_{nN}.$$
Consequently,
$|{f_*}^n\lambda-m|\geq |{f_*}^n\lambda[0,x_{nN})-m[0,x_{nN})|
\geq\frac{\delta}{2}\,x_{nN}
\geq\frac{\delta}{2}\,c_1\,\left(\textstyle\frac{1}{nN}\right)^{1/\alpha}$,
by Lemma~\ref{l.meas}~(i).

Now suppose $\liminf_{x\rightarrow 0}\frac{\lambda[0,x]}{x}>1$ and let
$\psi=\frac{d\lambda}{d m}$. Let $\lambda^\prime = \left(1-\frac{\psi-1}{|\psi-1|_\infty}\right)\,m$.
Then the proof of the first part of the lemma applies to $\lambda^\prime$ and
$|f_*^n\lambda^\prime-m|=|f_*^n\lambda-m|/|\psi-1|_\infty$.
\end{proof}

It is more delicate to obtain lower bounds on the decay rates of regular (ie: H\"older) functions.
One approach is to exploit symmetry of the cut function, when this is available.

We say that the cut function $\phi$ is symmetric if
\begin{equation}
\label{eqn:sym_cut}
1- \phi(t) = \phi(1-t) \textnormal{ for all } t \in [0,1]
\end{equation}
Equivalently, $\alpha =\alpha^\prime$, $c_0 = c_1$ and $g_0 = g_1$.

It follows that $a=\int_0^1 \phi (t)\, dt= 1/2$ and $x_n^\prime = 1-x_n$ for every $n$.
Note that Examples \ref{ex:alpha=1} and \ref{ex:cristadoro} satisfy this condition.

\begin{theorem}\label{t.lower_function}\textnormal{[Sharp decay rates for H\"older data]}
Suppose the cut function $\phi$ satisfies symmetry equation (\ref{eqn:sym_cut}).
Then
there are Lipschitz functions~$\varphi,\psi$ and a constant~$c_\alpha$
such that
$$\left|\int_0^1\varphi\circ f^n\,\psi\,dm -\int_0^1\varphi\,dm\int_0^1\psi\,dm\right|
\geq c_\alpha\,n^{-1/\alpha}.$$
\end{theorem}

The proof is in Appendix 2.

\section{Mixing rates IV -- polynomial decay of correlations for~$(B_\alpha,\m)$}\label{s:2ddecay}

Suppose $\varphi, \psi$ are two bounded measurable functions on a Borel probability space  $(X, p)$ and
$T$ is a measure preserving map on $X$. We write
\newcommand\Cor{\textnormal{Cor}}

$$\Cor_n(\varphi,\psi) = \left|\int_X\varphi\circ T^n\,\psi\,dp
-\int_X\varphi\,dp\,\int_X\psi\,dp\right|.$$

\begin{theorem}\label{th.2ddecay}
Let $\phi$ be a cut function as detailed in Section \ref{sec:family}, let  $B$
be the associated baker's transformation and set $\gamma= \max\{ \alpha, \alpha^\prime\}$.
If $\varphi$ and $\psi$ are H\"older continuous on $S$ then with respect to the measure $\m$ we have
$$\Cor_n(\varphi,\psi)=O(n^{-1/\gamma}).$$
The constant in the order notation depends on $\varphi,\psi$ and $\gamma$.
If $\phi$ satisfies the symmetry condition (\ref{eqn:sym_cut}),
 there are $\varphi,\psi$ for which this rate is sharp.
\end{theorem}

The proof proceeds in the expected fashion:
by applying the $1$-dimensional decay result for $f$
to suitably chosen $\varphi_0$ that depend only on the ``future'' (that is,
are $\varphi_0$ that are constant on vertical fibres). If $\varphi_0(x,y)$ depends
only on $x$ then $\overline{\varphi}=\varphi_0\circ\pi^{-1}$ has an unambiguous
definition (recall $\pi(x,y)=x$), and hence
\begin{equation}\label{eqn:corindpast}
\Cor_n(\varphi_0,\psi) = \left|\int_0^1\overline{\varphi}\circ f^n\,\overline{\psi}(x)\,dm
- \int_0^1\overline{\varphi}\,dm\,\int_0^1\overline{\psi}(x)\,dm\right|
\end{equation}
where $\overline{\psi}(x)= \int_0^1\psi(x,y)\,dm(y)$.

{\em Proof of (\ref{eqn:corindpast}):\/} Since
$\varphi_0(x^\prime,y^\prime)=\varphi_0(x^\prime,0)$ for each $(x^\prime,y^\prime)$
$$\varphi_0\circ B^n(x,y)=\varphi_0(f^n(x),g_n(x,y))=\varphi_0\circ\pi^{-1}(f^n(x))=\overline{\varphi}\circ f^n(x)$$
(see~(\ref{eqn:Bnformula})). Hence, by Fubini's theorem,
$$\int_S \varphi_0\circ B^n\,\psi\,d\m 
= \int_0^1\overline{\varphi}(f^n(x))\int_0^1 \psi(x,y)\,dm(y)\,dm(x)
=\int_0^1\overline{\varphi}\circ f^n\overline{\psi}\,dm.$$
Since,
$$\int_S\varphi_0\,d\m = \int_S\varphi_0 d(\pi_*m)=\int_0^1\varphi_0\circ\pi^{-1}\,dm
\qquad\mbox{and}\qquad
\int_S\psi\,d\m = \int_0^1\overline{\psi}\,dm$$
the proof is complete.\qquad$\Box$

It is evident that the lower bounds on the rate of correlation decay
obtained for $f$ in Theorem \ref{t.lower_function}
carry over to $B$: simply extend the one-dimensional functions to vertical fibres
by translation. Lifting the upper bounds requires more work, and exploits the
fact that for a H\"older continuous $\varphi$, $\varphi\circ B^n$ is very nearly constant
on ``most'' fibres when $n$ is large.

\begin{lemma}\label{lem:decompose} Let $\varphi$ be H\"older continuous on~$S$. Let
$B$ and $\gamma$ be as defined in Theorem \ref{th.2ddecay}. Then
there is a constant $C$ such that for
each sufficiently large $k$ there are functions $\varphi_0,\varphi_1,\varphi_2$
such that
$$\varphi\circ B^k = \varphi_0+\varphi_1+\varphi_2$$
where
\begin{itemize}
\item $\varphi_0$ is constant on vertical fibres and $|\varphi_0|_\infty\leq |\varphi|_\infty$,
\item $|\varphi_1|_\infty\leq k^{-1/\gamma}$ and
\item $|\varphi_2|_{L^1}\leq C\,|\varphi|_\infty\,k^{-1/\gamma}$.
\end{itemize}\end{lemma}

\begin{proof} Let $k$ be fixed.
We begin with some notation: let $\hat \Delta_0=[x_0,x^\prime_0)\times[0,1]\subset S$
(where $\{x_0,x^\prime_0\}$ is the period~$2$ orbit of~$f$ from Section \ref{s:towers})
and let
$$\beta=\sup_{x\in[x_0,x^\prime_0)}\max\{\phi(x),1-\phi(x)\}.$$
Then, when $B(x,y)\in\hat\Delta_0$, $\tilde{\phi}(x)\leq \beta$ (see equation~(\ref{eqn:lin_fibre_action})),
so vertical fibres are contracted by at least $\beta$ every time the orbit visits $\hat\Delta_0$.
If an orbit segment $\{B^n(x,y)~:~0\leq n<k\}$ has made at least $N$ visits to $\hat\Delta_0$ then
\begin{equation}\label{eqn:fibrecontraction}
|B^k(x,y)-B^k(x,y^\prime)| = |g_k(x,y)-g_k(x,y^\prime)| = {\partial_y g_k}\,|y-y^\prime|\leq \beta^N
\end{equation}
(again, see (\ref{eqn:lin_fibre_action}) and note that $0\leq \tilde\phi\leq 1$). If $\varphi$ is $\zeta$--H\"older
then there is a constant $C_\varphi$ such that
$|\varphi(x,y)-\varphi(x^\prime,y^\prime)|\leq C_\varphi |(x,y)-(x^\prime,y^\prime)|^\zeta$.
Choose $N$ such that $C_\varphi (\beta^N)^\zeta\leq k^{-1/\gamma}$. Clearly $N\approx\log k \ll k$.
Next, define a ``good set''
$$
G_k = \left\{(x,y)\in S~:~ B^{n_j}(x,y)\in \hat\Delta_0\textnormal{~for~}n_1<\cdots<n_N<k\right\}
$$
and put $\varphi_0(x,y)=(\varphi\circ B^k)(x,0)\mathbf{1}_{G_k}(x,y)$,
$\varphi_1=(\varphi\circ B^k)\mathbf{1}_{G_k}-\varphi_0$
and $\varphi_2=(\varphi\circ B^k)\,\mathbf{1}_{S\setminus G_k}$.

Since $\varphi_0$ takes only values of $\varphi$ (and $0$ outside $G_k$),
$|\varphi_0|_\infty\leq |\varphi|_\infty$. Moreover,
since $B^n(x,y)\in\hat\Delta_0$ if and only if $f^n(x)\in[x_0,x^\prime_0)$, $G_k$ is a union of vertical
fibres, so $\mathbf{1}_{G_k}(x,y)$ depends only on $x$. This establishes the claimed properties
of $\varphi_0$.

For $\varphi_1$, if $(x,y)\in G_k$ then $\{B^n(x,y)\}_{0\leq n\leq k}$ has made at least~$N$ visits
to $\hat \Delta_0$, so
$$|\varphi(B^k(x,y))-\varphi(B^k(x,y^\prime)| \leq C_\phi\,(\beta^N)^\zeta \leq k^{-1/\gamma}$$
by the H\"older property, (\ref{eqn:fibrecontraction}) and the choice of~$N$.

{\em Claim:\/} There are constants $c_1$ and $c_2$ (independent of $k$) such that for all large enough~$k$
$$\m\{S\setminus G_k\} \leq c_1\,k^{-1/\gamma}+c_2\,N^{2+1/\gamma}\,k^{-1-1/\gamma}.$$

{\em Proof of the lemma, given the claim:\/} All that remains is to control $\varphi_2$. Since
$N$ grows like $\log k$, taking $C=c_1+1$ gives $\m\{S\setminus G_k\} \leq C\,k^{-1/\gamma}$
for all large enough~$k$. The bound on $|\varphi_2|_{L^1}$ follows.

{\em Proof of claim:\/} Let $$\tau_1(x)=\min\{ n\geq 0~:~ B^n(x,y)\in\hat\Delta_0\}=\min\{n\geq 0~:~f^n(x)\in\Delta_0\}$$
and $\tau_{i+1}(x) = \tau_{i}(x)+R(f^{\tau_i(x)}(x))$ where $R$ is the usual
return time function to the ``base of the tower''~$\Delta_0$. Note that $f^{\tau_i}=(f^R)^{i-1}\circ f^{\tau_1}$.
Let $$H_k=\{x~:~\tau_1(x)\leq k/2\textnormal{~and~} \tau_{i+1}(x)-\tau_i(x)\leq k /2N,\quad i=1,\ldots,N-1\}.$$
Clearly, $H_k\times[0,1]\subset G_k$ so
\begin{eqnarray}\label{e.estGk1}
\m\{S\setminus G_k\} \leq m\{[0,1]\setminus H_k\} &\leq& m\{\tau_1>k/2\}
+\sum_{i=1}^{N-1} m\{\tau_{i+1}-\tau_i>k /2N\}\nonumber\\
&=&\sum_{j+1> k/2}m(J_j\cup J^\prime_j)\\
&&\qquad\qquad + \sum_{i=1}^{N-1}m\circ (f^R)^{-(i-1)}\{R\circ f^{\tau_1}>k/2N\}\nonumber
\end{eqnarray}
using $\tau_1|_{J_j^{(\prime)}}=j+1$ and the definition of $\tau_{i+1}$.
Next, $m|_{\Delta_0}$ is invariant under $f^R$, so
\begin{eqnarray}\label{e.estGk2}
\sum_{i=1}^{N-1}m\circ (f^R)^{-(i-1)}\{R\circ f^{\tau_1}>k/2N\}
&=&(N-1)\, m \{R\circ f^{\tau_1}>k/2N\}\nonumber \\
&=&(N-1)\,m \circ (f^{\tau_1})^{-1}\{D_k\}.
\end{eqnarray}
where $D_k=\{R>k/2N\} = \cup_{j+1>k/2N}(I_j\cup I_j^\prime)$. Note that $m(D_k)\approx (k/2N)^{-1-1/\gamma}$
(Lemma~\ref{l.meas}).
Since $f^{\tau_1}=id|_{[x_0,x^\prime_0]}+\sum_{j=0}^\infty f^{j+1}|_{J_j\cup J_j^\prime}$ and each branch of
$f^{\tau_1}$ has uniformly bounded distortion (see proof of Lemma~\ref{l:distortion}),
there is a constant $c\geq 1$ such that
\begin{eqnarray}\label{e.estGk3}
m\circ(f^{\tau_1})^{-1}\{D_k\}
&\leq& m(D_k)+c\,\sum_{j=0}^\infty\frac{m(D_k)}{m(\Delta_0)}(m(J_j)+m(J_j^\prime))\nonumber\\
&\leq& c\,\frac{m(D_k)}{m(\Delta_0)}\leq c^{\prime} (k/2N)^{-1-1/\gamma}.
\end{eqnarray}
Combining (\ref{e.estGk1}), (\ref{e.estGk2}), (\ref{e.estGk3}) and the estimate
$\sum_{j+1> k/2}m(J_j\cup J^\prime_j)\approx (k/2)^{-1/\gamma}$ from Lemma~\ref{l.meas}
completes the proof.
\end{proof}

{\em Proof of Theorem~\ref{th.2ddecay}:\/} First,
$\overline{\psi}$ inherits the H\"older property from~$\psi$.
Put $n^\prime=\lfloor n/3\rfloor$,  $k=n-n^\prime$ and decompose
$$\varphi\circ B^k = \varphi_0+\varphi_1+\varphi_2$$
as in Lemma~\ref{lem:decompose}. Then,
$$\Cor_n(\varphi,\psi)= \Cor_{n^\prime}(\varphi\circ B^k,\psi)
\leq \Cor_{n^\prime}(\varphi_0,\psi) + \sum_{i=1}^2\Cor_{n^\prime}(\varphi_i,\psi).$$
The latter two terms are bounded above by $C\,n^{-1/\gamma}$ for some constant $C$
independent of $n$ and the first term is
$O((n^\prime)^{-1/\gamma})=O(n^{-1/\gamma})$ by~(\ref{eqn:corindpast})
and Theorem~\ref{endomorphism}~part~3.
\qquad$\Box$

\section*{Appendix 1: precise distortion and decay estimates}

Assume that $\alpha,~\alpha^\prime,~c_0,~c_1,~g_0$ and $g_1$ are given,
defining $\phi$ as in Section \ref{sec:family}, the generalized baker's transformation
$B$ and two branched expanding map $f$.  As noted in Equation \ref{eqn:DB} we
compute
$${f}'(x) = \left\{\begin{array}{ll}
	\textstyle\frac{1}{\phi(f(x))}& x<a,\\
	\textstyle\frac{1}{1-\phi(f(x))}& x>a.
\end{array}\right.$$
From the expression for $\phi$, estimates on $g_0$ and
the expression
$$
x - {f}^{-1}(x) = \int_0^x(1-\phi(t))\,dt,
$$
valid under the left branch of $f$,
we obtain constants $C_0$, $\delta_0>0$ such that for
all  $0 \leq x \leq \delta_0$ we have
\begin{equation}
\label{e.deltax_at_zero}
C_0^{-1} x^{1+\alpha} \leq x - {f}^{-1}(x) \leq C_0 x^{1+\alpha}.
\end{equation}
A similar estimate holds for $x$ near 1 using the right branch of
$f$:  There exists a constant $C_1$ and $\delta_1 > 0$
such that for all $1-\delta_1 \leq x \leq 1$
\begin{equation}
\label{e.deltax_at_one}
C_1^{-1} (1-x)^{1+\alpha^\prime} \leq {f}^{-1}(x) - x \leq C_1 (1-x)^{1+\alpha^\prime}
\end{equation}
Continue with the notation $x_0$ the left most period--$2$ point,
$x_k={f}^{-1}(x_{k-1})\cap [0,x_k)$ and
similarly for $x^\prime_k$.

\subsection*{Proof of Lemma \ref{l.meas} on asymptotics of the $x_n, ~x_n^\prime$}
\label{a:1}

(i) We first establish the estimates on $x_n$.  First, for any $y\geq \delta^{-1},\,z \geq 0$,
the mean value theorem and~(\ref{e.deltax_at_zero})  give
\begin{eqnarray}
\frac{[\frac{1}{y}]^{1/\alpha}-[\frac{1}{y+z}]^{1/\alpha}}
{[\frac{1}{y}]^{1/\alpha}-{f}^{-1}([\frac{1}{y}]^{1/\alpha})}
&\leq&\textstyle\frac{C_0}{\alpha}\,
\left[\frac{1}{y+\theta\,z}\right]^{1/\alpha-1}
\left(\frac{1}{y}-\frac{1}{y+z}\right)\,y^{1+1/\alpha} \nonumber\\
&=&\textstyle\frac{C_0}{\alpha}\,
\left[\frac{y}{y+\theta\,z}\right]^{1/\alpha}
\left[\frac{y+\theta\,z}{y+z}\right]\,z
\label{e.cflacunary}\end{eqnarray}
(where $\theta\in[0,1]$).
The upper and lower bounds are obtained by distinct
applications of~(\ref{e.cflacunary}).
First, fix $n$ such that $x_n^{-\alpha} < \delta_0$ and
set $y = x_n^{-\alpha}$ and $z= \left[\frac{C_0}{\alpha}\right]^{-1}$.

 Then the RHS  of Equation \ref{e.cflacunary} is bounded above
by~$1$, so that
$$
\textstyle[\frac{1}{y}]^{1/\alpha}-[\frac{1}{y+z}]^{1/\alpha}
\leq [\frac{1}{y}]^{1/\alpha}-{f}^{-1}([\frac{1}{y}]^{1/\alpha}).
$$
In particular,
${f}^{-1}([\frac{1}{y}]^{1/\alpha})
\leq[\frac{1}{y+z}]^{1/\alpha}$, so that
by using $y={x_n}^{-\alpha}$ and induction,
for all $k\geq 0$,
$$\textstyle x_{n+k}={f}^{-k}(x_n)=
{f}^{-k}([\frac{1}{y}]^{1/\alpha})
\leq\left[\frac{1}{y+k\,z}\right]^{1/\alpha}
\leq \frac{1}{z^{1/\alpha}}\,
\left[\frac{1}{k}\right]^{1/\alpha} \approx \left[ \frac{1}{n+k}\right]^{1/\alpha}.$$
On the other hand, whenever $y\geq z$ then the RHS of~(\ref{e.cflacunary}) is
bounded below by
$\frac{1}{\alpha\,C_0}\frac{1}{2^{1+1/\alpha}}\,z$. Pick
$z=C_0 \alpha 2^{1+1/\alpha} $ and set
$y=\max\{z,{x_n}^{-\alpha}\}$.   Then
$$\textstyle x_{n+k}={f}^{-k}(x_n)\geq
{f}^{-k}([\frac{1}{y}]^{1/\alpha})
\geq\left[\frac{1}{y+k\,z}\right]^{1/\alpha}
\geq \frac{1}{(2\,y)^{1/\alpha}}\,
\left[\frac{1}{k}\right]^{1/\alpha} \approx \left[\frac{1}{n+k}\right]^{1/\alpha}.$$
This establishes the asymptotics for the $x_k$.  The
estimates on $x_k^\prime$ are similar, using $\alpha^\prime$ instead of $\alpha$
and Equation (\ref{e.deltax_at_one}) instead of Equation (\ref{e.deltax_at_zero}).
\newline
(ii) Since $J_k=[x_{k+1},x_k)$, we have
$m(J_k)=x_k-x_{k+1}\approx {x_k}^{1+\alpha}\approx
\left[\frac{1}{k}\right]^{1+1/\alpha}$ by~(\ref{e.deltax_at_zero}) and part~(i) of
the lemma. The estimate on $J_k^\prime$ using $x_k^\prime$ is similar. \newline
(iii) Observe that on $(a,x^\prime_0)$, $f^\prime > 1$ is decreasing
so for $x \in I_k := [t_{k+1}, t_k]$ we have
$f^\prime (t_{k+1}) \geq f^\prime(x) \geq f^\prime(t_{k})$.
But, by part (i), for all sufficiently large $k$,
$$
f^\prime(t_k) = (1 - \phi(x_{k-1}))^{-1}
\approx \left( (k-1)^{\frac{1}{\alpha}}\right)^\alpha \approx k
$$
The argument for intervals $I_k^\prime$ in $[x_0, a)$ is similar.
\newline
(iv) Since $f:I_k\rightarrow J_{k-1}$ bijectively, there is an $x\in I_k$
such that
$$m(I_k) = \frac{m(J_{k-1})}{{f_\alpha^\prime}(x)}
\approx \left[\textstyle\frac{1}{k-1}\right]^{1+1/\alpha}\,
\textstyle\frac{1}{k} \approx
\left[ \textstyle\frac{1}{k}\right]^{2 + \textstyle\frac{1}{\alpha}}$$
using (ii) and (iii). The argument for the $I_k^\prime$ is similar. \newline

(v) When $n\leq\rho\,k$, $[\frac{1}{k+n}]\approx [\frac{1}{k}]$ so the
estimate follows from parts~(i) and~(ii) and the fact that
$x_k-x_{k+n}=\sum_{k\leq i<k+n}m(J_i)$.\hfill$\Box$

\subsection*{Proof of Lemma \ref{l:distortion} on uniform distortion}

We assume that $y,z\in I_i\subset\Delta_{0,i} \subseteq (a, x_0^\prime)$.  The
case where $y,z\in{I_i}'$ is similar.   For each $1\leq k<i+1=R$ let
$y_k=f^{R-k}(y)$ and $z_k=f^{R-k}(z)$.
Thus $y_k,z_k\in J_{k-1}$. Now,
$$[\log({f}')]'|_{J_k} =
\left.\textstyle\frac{{f}''}{{f}'}\right|_{J_k} =
\left[\textstyle\frac{-{\phi}'}{{\phi}^2}\right]
\circ f|_{J_k}
\approx \left(\left[\textstyle\frac{1}{k+1}\right]^{1/\alpha}\right)^{\alpha-1}$$
The final estimate in this expression follows from two observations. First
note that $\phi|_{f(J_j)} \geq \phi(x^\prime_0) >0$,
providing a uniform lower bound on the denominator for all  $j= 0,1, \dots$
and second,
$-{\phi}'\circ f(x)=\alpha c_0 [f(x)]^{\alpha-1} + g^\prime_0(f(x))
\approx {[f(x)]}^{\alpha-1} \approx {x}^{\alpha-1}$ whenever $x\in [0, x_0]$
since $x \leq f(x) \leq 2x$. Thus,
\begin{eqnarray}\label{e.bd1}
\left|\log\textstyle\frac{{f}'(y_k)}{{f}'(z_k)}\right|
\leq c\,\left[\textstyle\frac{1}{k}\right]^{1-1/\alpha}\,|y_k-z_k|
&=& c\,\left[\textstyle\frac{1}{k}\right]^{1-1/\alpha}\,m(J_{k-1})\,
\textstyle\frac{|y_k-z_k|}{m(J_{k-1})}\nonumber \\
&\leq& c'\,\left[\textstyle\frac{1}{k}\right]^2\,
\textstyle\frac{|y_k-z_k|}{m(J_{k-1})}
\leq c'\,\left[\textstyle\frac{1}{k}\right]^2
\end{eqnarray}
since $|y_k-z_k|\leq m(J_{k-1})\approx m(J_k)$, where the latter estimate
uses Lemma~\ref{l.meas}~(ii).

A slightly different computation is required
for the first iterate.
$$[\log({f}')]'|_{I_i} =
\left.\textstyle\frac{{f}''}{{f}'}\right|_{I_i} =
\left[\textstyle\frac{{\phi}'}{{[1-\phi]}^2}\right] \circ f|_{I_i}$$
Therefore, for some $t$ in $I_i$ between $y$ and $z$ we have
\begin{equation}
\left|\log\textstyle\frac{{f}'(y)}{{f}'(z)}\right|
= \textstyle\frac{|{\phi}'(f(t))|}{{[1-\phi(f(t))]}^2}| y-z| \approx
\frac{m(I_i)}{m(J_{i-1})} \frac{|y-z|}{m(I_i)}
\end{equation}
Here we have used $1-\phi \approx x^\alpha$, for $x \approx 0$
$|\phi'(x)|\approx x^{\alpha -1}$, $f(t) \in J_{i-1}$, (hence
$f(t) \approx \left( \frac{1}{i-1} \right)^{\frac{1}{\alpha}}$) and estimate
(ii) from Lemma~\ref{l.meas}.
Next, observe that for some $t_0 \in J_{i-1}$
\begin{equation}
\textstyle \frac{m(I_i)}{m(J_{i-1})}  =
\textstyle \frac{ 1}{m(J_{i-1})} \int_{J_{i-1}} 1 - \phi = 1- \phi(t_0)
\approx \textstyle \frac{1}{i-1} \approx \textstyle \frac{1}{i}
\end{equation}
since then $t_0 \approx \left(\textstyle \frac{1}{i-1}\right)^{\frac{1}{\alpha}}$.  Therefore
\begin{equation}\label{e.bd2}
\left|\log\textstyle\frac{{f}'(y)}{{f}'(z)}\right|
\leq \textstyle \frac{c''}{i} \textstyle\frac{|y-z|}{m(I_i)} \leq \textstyle \frac{c''}{i}
\end{equation}
for some $c''$ independent of $y,z,i$ (but possibly depending on $\alpha$).

Now, since $({f}^R)'(y) = {f}'(y)\,{f}'(y_{R-1})\,
\cdots\,{f}'(y_1)$ (and similarly for $z$),
\begin{equation}
\label{e.bd3}
\left|\log\textstyle\frac{({f}^R)'(y)}{({f}^R)'(z)}\right|
= \left|\log\textstyle\frac{{f}'(y)}{{f}'(z)}\right|+
\sum_{k=1}^i\left|
\log\textstyle\frac{{f}'(y_k)}{{f}'(z_k)}\right|
< \frac{c''}{i} + c'\,\textstyle\sum_{k=1}^\infty\frac{1}{k^2}
\leq c'' + c'\,\textstyle\sum_{k=1}^\infty\frac{1}{k^2}
\stackrel{\mbox{\tiny def}}{=} C.\end{equation}
Now put $D=e^C$. Since the inequality in (\ref{e.bd3}) holds uniformly for any choice of
$y,z\in I_i$ and the map $f^R:I_i\rightarrow\Delta_0$ is
bijective, we have
$$\textstyle\frac{|y-z|}{m(I_i)}
\leq D\,\textstyle\frac{|f^R(y)-f^R(z)|}{m(\Delta_0)}.$$
Similarly, $\frac{(f^k)'(y_k)}{(f^k)'(z_k)}\leq D$
and since $f^k(y_k)= f^R(y)$ and $f^k(z_k)= f^R(z)$,
$$\textstyle\frac{|y_k-z_k|}{m(J_{k-1})}
\leq D\,\textstyle\frac{|f^R(y)-f^R(z)|}{m(f^k(J_{k-1}))}
= D\,\textstyle\frac{|f^R(y)-f^R(z)|}{m(\Delta_0)}.$$
The last two
displayed expressions can now be used to refine (\ref{e.bd2}) and
(\ref{e.bd1}), yielding
$$
\left|\log\textstyle\frac{{f}'(y)}{{f}'(z)}\right|\leq
\textstyle\frac{c''}{i}\,D\,\textstyle\frac{|f^R(y)-f^R(z)|}{m(\Delta_0)}
\mbox{~and~}
\left|\log\textstyle\frac{{f}'(y_k)}{{f}'(z_k)}\right|\leq
c'\,[\textstyle\frac{1}{k}]^2\,D\,\textstyle\frac{|f^R(y)-f^R(z)|}
{m(\Delta_0)}
$$
from which:
$$\left|\log\textstyle\frac{({f}^R)'(y)}{({f}^R)'(z)}\right|
\leq C\,D\,\textstyle\frac{|f^R(y)-f^R(z)|}{m(\Delta_0)}.$$
Finally, if $|\log x|<C$ then $|\log x|>\frac{C}{e^C-1}\,|x-1|$
by an elementary convexity estimate. In view
of~(\ref{e.bd3}),
$$\left|\textstyle\frac{(f^R)'(y)}{(f^R)'(z)}-1\right| \leq \textstyle\frac{D-1}{C}\,\left|\log\textstyle\frac{(f^R)'(y)}{(f^R)'(z)}\right|
\leq \textstyle\frac{D\,(D-1)}{m(\Delta_0)}|f^R(y)-f^R(z)|. \qquad\qquad \Box$$

\section*{Appendix 2:  Lower bounds for H\"older observables}

A function $\psi:[0,1]\rightarrow\R$ will be called {\em anti-symmetric\/} if
$\psi(1-x)=-\psi(x)$ for each $x\in[0,1]$.

\begin{lemma}\label{lem:slowmix}
Let $\phi$ be a cut function satisfying symmetry condition
(\ref{eqn:sym_cut}) and let $f$
denote the expanding 1-D  expanding map determined by  $\phi$ via
(\ref{eqn:map1}). Suppose that $\psi$ is decreasing and anti-symmetric. Then
$\frac{d}{dm}{f_*}^n(\psi\,m)$ is decreasing and anti-symmetric for each $n>0$.
\end{lemma}

\begin{proof}
First, let $L$ be the Frobenius--Perron (transfer) operator for $f$,
so $\frac{d}{dm}{f_*}^n(\psi\,m)=L^n\psi$. By induction, it
suffices to show that $L\psi$ has the required properties.
Next, since the cut-function $\phi$ satisfies Equation (\ref{eqn:sym_cut}) for each
$t\in[0,1]$, the transformation $f$ satisfies $f(1-x) = 1-f(x)$ for each $x\neq 1/2$.
Let $L_-$ be the Frobenius-Perron operator for $x\mapsto (1-x)$, so
$LL_-=L_-L$ and $L_-\psi=-\psi$. Then
$$L\psi(1-x) = L_-L\psi(x) = LL_-\psi(x)=L(-\psi)(x)=-L\psi(x).$$
Next, since $\psi(1/2)=-\psi(1/2)$, $\psi(1/2)=0$ and therefore
$\psi\mathbf{1}_{(0,1/2)}\geq 0\geq \psi\mathbf{1}_{(1/2,1)}$
(and also $L\psi(1/2)=0$).
Since $\phi$ is a decreasing function, $1/f^\prime=\phi\circ f$ is decreasing
on $(0,1/2)$, so $\psi_1:=L(\psi\,\mathbf{1}_{(0,1/2)})$ is decreasing. A similar
argument shows that $\psi_2:=L(\psi\,\mathbf{1}_{(1/2,1)})$ is decreasing,
so $L\psi=\psi_1+\psi_2$ is decreasing.
\end{proof}

{\em Proof of Theorem~\ref{t.lower_function}:}
Let
$\varphi(x)=\psi(x)=x$ and put $\lambda=m + (\psi-1/2)\,m$. Then $\lambda$
is a probability measure and since $\int\varphi\,dm=1/2$,
$$\int(\varphi-1/2)d(f_*^n\lambda) = \int (\varphi-1/2)\circ f^n\,d\lambda
=\int \varphi\circ f^n\,\psi\,dm - \int\varphi\,dm\int\psi\,dm.$$
Now, $f_*^n\lambda = m-(L^n(1/2-\psi))\,m$ where
$L$ is the Frobenius--Perron operator for $f$, so the previous equation can be rewritten as
\begin{equation}\label{eqn:1dDoC1}
\int(1/2-\varphi)\,L^n(1/2-\psi)\,dm = \int \varphi\circ f^n\,\psi\,dm - \int\varphi\,dm\int\psi\,dm.
\end{equation}
By Lemma~\ref{lem:slowmix}, $L^n(1/2-\psi)$ is decreasing and antisymmetric
(and in particular is non-negative on $(0,1/2)$, non-positive on
$(1/2,1)$). Hence, $(1/2-\varphi)\,L^n(1/2-\psi)\geq 0$ and so
\begin{equation}
\label{eqn:1dDoC2}
\begin{array}{rl}
\int_0^1 (1/2-\varphi)\,L^n(1/2-\psi)\,dm &\geq \int_0^{1/4} (1/2-\varphi)\,L^n(1/2-\psi)\,dm\\
&\geq \frac{1}{4}\int_0^{1/4} L^n(1/2-\psi)\,dm\\
&\geq \frac{1}{4}\,\frac{1}{2}\int_0^{1/2} L^n(1/2-\psi)\,dm\\
&= \frac{1}{4}\,\frac{1}{4}\int_0^1 |L^n(1/2-\psi)|\,dm=\frac{1}{16}|f_*^n\lambda-m|
\end{array}
\end{equation}
(the last equality follows by the definition of $\lambda$).
Clearly, $\lambda$ is separated from $m$ at $0$, so the theorem follows from
equations (\ref{eqn:1dDoC1}), (\ref{eqn:1dDoC2}) and Theorem~\ref{t.lower_measure}.\qquad $\Box$


\begin{thebibliography}{10}



\bibitem{A-Y} Alexander, J.~and Yorke, J., \textit{Fat baker's transformations.}\/
Ergodic Theory Dynam.~Systems,~ {\bf 4}(1984), 1--23.

\bibitem{A1}
J~Alves., \textit{
SRB measures for non-hyperbolic systems with multidimensional expansion.}\/
 Ann.~Sci.~\'Ec.~Norm.~Sup\'er.~(4), 33(1)(2000), 1--32.


\bibitem{A} Alves, J.,  \textit{A survey of recent results on some statistical features of
	non-uniformly expanding maps.}\/
Discrete~Contin.~Dyn.~Syst., {\bf 15} \#1(2006), 1--20.


\bibitem{AA} Alves, J.~and Ara\'ujo, V.,  \textit{Hyperbolic times: frequency versus integrability.}\/
  Ergodic Theory Dynam.~Systems~ {\bf 24}(2004), 329--346.


\bibitem{ABV} Alves, J.,~ Bonatti, C.~and Viana, M., \textit{SRB measures for partially hyperbolic
systems with mostly expanding central direction.}\/
Invent.~Math. {\bf 140}(2000), 351--398.

\bibitem{ALP} Alves, J.,~ Luzzatto, S.~and Pinheiro, V., \textit{Markov structures and decay of
correlations for non-uniformly expanding dynamical systems.}\/
Ann.~Inst.~H.~Poincar\'e Anal.~Non Lin\'eaire {\bf 22} \#6 (2005), 817--839.

\bibitem{ALP2} Alves, J.,~ Luzzatto, S.~and Pinheiro, V., \textit{Lyapunov exponents and rates of mixing for one-dimensional maps.}\/
Ergodic Theory Dynam.~Systems~ {\bf 24}(2004), 637--657.

\bibitem{B-M} Bose, C., and Murray, R., \textit{Integrability of first hyperbolic times for intermittent maps with unbounded derivative.}\/  Preprint: 2011.


\bibitem{B-Y} Benedicks, M.~and Young, L-S., \textit{Markov extensions and decay of correlations for certain H\'enon maps.}\/ G\'eom\'etrie complexe et syst\`emes dynamiques (Orsay, 1995). Ast\'erisque  \#26, xi(2000), 13--56.

\bibitem{B1} Bose, C., \textit{Generalized baker's transformations.}\/  Ergodic Theory Dynam.~Systems~
{\bf 9}(1989), 1--17.


\bibitem{C-Y}Chernov, N.I.~and Young, L-S., \textit{
Decay of correlations for Lorentz gases and hard balls.}\/ (English summary) Hard ball systems and the Lorentz gas.
Encyclopaedia Math.~Sci.~  {\bf 101} Springer Verlag,  2000, pp 89--120.

\bibitem{CZ} Chernov, N.~and Zhang, H-K., \textit{ Improved estimates for correlations in billiards.}\/   Comm.~Math.~Phys.~{\bf 277\/} \#2(2008), 305--321.



\bibitem{C-H-V} Cristadoro, G.,~ Haydn, N.~and Vaienti, S., \textit{Statistical properties of intermittent maps
 with unbounded derivative.}\/  Nonlinearity {\bf 23} \#5(2010), 1071--1095.



\bibitem{G1} Gou\"ezel, S., \textit{Sharp polynomial estimates for the decay of correlations.}\/
 Israel~J.~Math.~ {\bf 139}(2004), 29--65.



\bibitem{K} Kwon, DY., \textit{The natural extensions of $\beta-$tranformations which
generalize baker's transformations.}\/ Nonlinearity {\bf 22} \#2 (2009), 301--310.



\bibitem{KS} Katok, A.~and Strelcyn, J.M., \textit{ Invariant Manifolds, Entropy and
Billiards; Smooth Maps with Singularities.}\/  Lecture Notes in Mathematics.~ {\bf1222}
Springer Verlag, 1986.

\bibitem{LSV} Liverani, C.,~ Saussol, B.~and Vaienti, S., \textit{A probabilistic approach to intermittency.}\/  Ergodic Theory Dynam.~Systems~  {\bf 19}(1999),  671--685.


\bibitem{Ma} Markarian, R., \textit{ Billiards with polynomial decay of correlations.}\/  Ergodic Theory Dynam.~Systems~ {\bf 24}(2004), 177--197.

\bibitem{P} Pianigiani, G., \textit{ First Return Map and Invariant Measures.}\/  Israel~J.~Math.~
{\bf 35}(1980), 32--48.

\bibitem{R} Rahe, M., \textit{ On a class of generalized baker's transformations.}\/
Canad.~J.~Math.~ {\bf 45}(1993), 638--649.

\bibitem{Rms} Rams, M., \textit{Absolute continuity of the  SRB measure for non-linear fat baker maps.}\/
 Nonlinearity {\bf 16}  \#5(2003),1649--1655.


\bibitem{S}  Sarig, O., \textit{ Subexponential decay of correlations.}\/
 Invent.~Math.~ {\bf 150}(2002), 629--653.

\bibitem{T} Tsujii, M., \textit{Fat solenoidal attractors.}\/ Nonlinearity  {\bf 14}(2001), 1011--1027.

\bibitem{V} Viana,  M., \textit{ Multidimensional non-hyperbolic attractors.}\/
Publ.~Math.~Inst.~Hautes \'Etudes Sci.~{\bf 85}(1997),
63--96.

\bibitem{Y1} Young, L-S., \textit{ Statistical Properties of Dynamical Systems with
Some Hyperbolicity.}\/  Ann.~of Math.~ {\bf 147} \#3(1998), 585-650.

\bibitem{Y2} Young, L-S., \textit{  Recurrence times and rates of mixing.}\/ Israel~ J.~Math.~ {\bf 110}(1999),
153--188.

\bibitem{Z} Zweim\"uller, Roland, \textit{Ergodic structure and invariant densities of non-Markovian interval maps with indifferent fixed points.} \/ Nonlinearity {\bf 11}(1998), no. 5, 1263--1276.



\end{thebibliography}
\end{document}